\documentclass[a4paper,oneside,12pt]{article}
\usepackage[latin2]{inputenc}
\usepackage[T1]{fontenc}
\usepackage[includehead,includefoot,margin=1in,top=1in,bottom=1in]{geometry}
\usepackage{amsmath,amssymb,amsthm,bbm}
\usepackage{titlesec,makeidx}
\usepackage[comma,round]{natbib}

\usepackage{xcolor}

\theoremstyle{plain}
\newtheorem{theorem}{Theorem}

\newtheorem{proposition}[theorem]{Proposition}
\newtheorem{corollary}[theorem]{Corollary}
\theoremstyle{definition}

\theoremstyle{remark}
\newtheorem{remark}{Remark}

\newcommand{\Z}{\ensuremath\mathbb{Z}}

\newcommand{\R}{\ensuremath\mathbb{R}}

\newcommand{\cC}{\ensuremath\mathcal{C}}
\newcommand{\cD}{\ensuremath\mathcal{D}}
\newcommand{\cF}{\ensuremath\mathcal{F}}

\newcommand{\cS}{\ensuremath\mathcal{S}}
\newcommand{\cV}{\ensuremath\mathcal{V}}
\newcommand{\cZ}{\ensuremath\mathcal{Z}}
\newcommand{\bA}{\ensuremath\mathbf{A}}
\newcommand{\bc}{\ensuremath\mathbf{c}}

\newcommand{\be}{\ensuremath\mathbf{e}}

\newcommand{\bM}{\ensuremath\mathbf{M}}

\newcommand{\bS}{\ensuremath\mathbf{S}}
\newcommand{\bx}{\ensuremath\mathbf{x}}
\newcommand{\bX}{\ensuremath\mathbf{X}}
\newcommand{\by}{\ensuremath\mathbf{y}}
\newcommand{\bY}{\ensuremath\mathbf{Y}}
\newcommand{\bU}{\ensuremath\mathbf{U}}
\newcommand{\bv}{\ensuremath\mathbf{v}}
\newcommand{\bV}{\ensuremath\mathbf{V}}
\newcommand{\bz}{\ensuremath\mathbf{z}}
\newcommand{\bZ}{\ensuremath\mathbf{Z}}

\newcommand{\bxi}{\ensuremath\boldsymbol{\xi}}

\newcommand{\bfeta}{\ensuremath\boldsymbol{\eta}}

\newcommand{\bnull}{{\ensuremath\boldsymbol{0}}}
\newcommand{\bfone}{\ensuremath\boldsymbol{1}}
\newcommand\1{\ensuremath\mathbbm{1}}

\newcommand\co[1]{\,\smash{\mathop{\longrightarrow}\limits^{#1}}\,}
\newcommand\eq[1]{\,\smash{\mathop{=}\limits^{#1}}\,}

\newcommand\myvector[2]{\left[\begin{array}{c} #1 \\ #2 \end{array}\right]}

\title{Ergodic properties of subcritical multitype Galton--Watson processes}
\author{G\'abor Sz\H ucs\\University of Szeged}
\date{}

\begin{document}

\maketitle

\begin{abstract}
In the paper the ergodic properties of multitype Galton--Watson processes are investigated in the subcritical case. A sufficient and necessary moment condition for the existence of an invariant distribution is proved without further regularity assumptions. Under some moment conditions geometric ergodicity is shown and rate of converge for the means of functions of the process is provided. The geometric properties of the Markovian class structure are also investigated.
\end{abstract}

\section{Introduction}
\label{sec:intro}

Galton--Watson processes are historically one of the oldest fields in the theory of stochastic processes. Although the probability of extinction was determined in the 19. century, the investigation of the asymptotic properties of branching processes was less intense during the first half of the 20. century. Finally, \citet{harris48} rediscovered the subject by showing the convergence of the scaled process in almost sure sense under some regularity conditions. The existence of a stationary distribution was proved in the 1970's. First, \citet{foster71} proved that the single type process possesses a unique invariant distribution in the subcritical case, \citet{kaplan73} achieved similar result for multitype processes with positive regular mean matrices.

In this paper the investigate the ergodic properties of multitype Galton--Watson processes. The research is limited to the subcritical case, but we do not assume any further regularity conditions. In our first theorem we investigate the Markovian class structure of the process. It is shown that there is a unique positive recurrent class which is reached with probability 1 in case fo any initial distribution. Using this result we can prove that Kaplan's logarithmic moment condition is sufficient and necessary for the existence of a stationary distribution in case of an arbitrary subcritical Galton--Watson process. In our last theorem we prove geometric ergodicity and provide a rate of convergence for the moments of functions of the process.

\section{Main results}
\label{sec:main results}

Let $\Z_+$ stand for the set of nonnegative integers, and consider an arbitrary positive integer $p$. The $p$-type Galton--Watson process $\bX_n=(X_{n,1},\dots,X_{n,p})^\top$, $n\in\Z_+$, is a $\Z_+^p$-valued Markov chain defined by the recursion
\begin{equation}\label{eq:GW def}
\bX_n=\sum_{k=1}^{X_{n-1,1}} \bxi_1(n,k) + \cdots + \sum_{k=1}^{X_{n-1,p}} \bxi_p(n,k) + \bfeta(n)\,,
\qquad
n=1,2,\dots,
\end{equation}
where the $\Z_+^p$-valued random vectors
\begin{equation}\label{eq:ind variables}
\bX_0, \bxi_i(n,k), \bfeta(n),
\qquad
i=1,\dots,p\,,
\quad
n,k=1,2,\dots
\end{equation}
are independent of each other, the offspring variables $\bxi_i(n,k)$, $n,k=1,2,\dots$, are identically distributed for every $i=1,\dots,r$, and the innovation variables $\bfeta(n)$, $n=1,2,\dots$, are identically distributed. In the following we interpret the random vector $\bX_n$ as the size of the $n$-th generation of an underlying population having $p$ different types of members. The offsprings of any subpopulation of the process are called \textit{1st generation offsprings}, and by recursion, the \textit{$n$-th generation offsprings} are defined as the offsprings of the $(n-1)$-th generation offsprings. We say that a member is \textit{multigeneration offspring} if it is $n$-th generation offspring with some positive integer $n$.

Throughout the paper we assume that the offspring variables have finite expectations, and we consider the mean matrix
\[
\bM:=\big[ E\bxi_1(1,1),\dots, E\bxi_p(1,1)\big]^\top\,.
\]
It is well-know that the asymptotic properties of the Galton--Watson process depends largely on the spectral radius $\varrho(\bM)$ of the matrix $\bM$. (See \citet{mode71} or \citet{athreya72}, for example.) The process is called subcritical, critical or supercritical if the spectral radius is smaller than 1, equal to 1 or larger than 1, respectively. In our paper we investigate only the subcritical (also known as stable) case. Note that in this case the multigeneration offsprings of any member of the population die out in finitely many steps with probability 1.

Consider an arbitrary $\bx=(x_1,\dots,x_p)^\top\in\Z_+^p$. Throughout the paper the notations $P_\bx$ and $E_\bx$ mean probability and expectation with respect to the condition $\{\bX_0=\bx\}$. We introduce the variable
\begin{equation}\label{eq:S def}
\bS(\bx)=\sum_{k=1}^{x_1} \bxi_1(1,k) + \cdots + \sum_{k=1}^{x_p} \bxi_p(1,k) \,,
\end{equation}
which is the number of the 1st generation offsprings of the initial population $\bX_0$ under the condition $\{\bX_0=\bx\}$. Note that we $E \bS(\bx)=\bM^\top\bx$, which implies by recursion that the expected number of the $n$-th generation offsprings of a population of size $\bx$ is $(\bM^n)^\top\bx$.

For any $\bx=(x_1,\dots,x_p)^\top\in\R^p$ and $\by=(y_1,\dots,y_p)^\top\in\R^p$ the notation $\bx\leq\by$ is understood componentwise, that is, $\bx\leq\by$ if and only if $x_i\leq y_i$ for $i=1,\dots,p$. The norm of the vector $\bx$ is defined as $\|\bx\|=|x_1|+\cdots+|x_p|$. By using the Kronecker delta symbol $\delta_{i,j}$ the system $\be_i=(\delta_{i,1},\dots,\delta_{i,p})^\top$, $i=1,\dots,p$, stands for the canonical basis of the vector space $\R^p$. In case of an arbitrary event $A$ the random variable $\1_A$ stands for the indicator of $A$, and for a set $B\subseteq\R^p$ the function $\1_B(\bx):=\1_{\{\bx\in B\}}$, $\bx\in\R^p$, is the indicator of $B$.

Our first result is a statement about the class structure of the chain $\bX_n$, $n\in\Z_+$.

\begin{theorem}\label{thm:irreducibility}
If a $p$-type Galton--Watson process is subcritical, then it has an aperiodic communication class $\cC\subseteq\Z_+^p$ such that the process reaches $\cC$ in finitely many steps with probability 1 for any initial distribution.
\end{theorem}

In our next theorem we provide a necessary and sufficient condition for the existence of a stationary distribution of subcritical Galton--Watson processes. Note that such a statement is already proved under the condition that the mean matrix $\bM$ is positive regular, meaning that there exists a positive integer $n$ such that all entries of $\bM^n$ are strictly positive. In this case by the well-known result of \citet{kaplan73} a stationary distribution exists if and only if the sum $\sum_{k=1}^\infty \log k P(\|\bfeta(1)\|=k)$ is finite. However, this equivalence is not true in the case of arbitrary offspring distributions. For example, if $\bM=\bnull$ then $\bX_n=\bfeta(n)$ for every positive $n$, and the distribution of the innovation variables is a stationary distribution for the process without any additional condition.

Let $M^{(n)}_{i,j}$ stands for the $(i,j)$-th entry of the matrix $\bM^n$, which is the expected number of those $n$-th generation offsprings of an arbitrary member of type $i$ which are of type $j$. We define $I$ as the set of those types $i=1,\dots,p$ for which there exists a type $j$ and integers $m_0\geq 0$ and  $m\geq 1$ such that $M^{(m_0)}_{i,j}>0$ and $M^{(m)}_{j,j}>0$. Since the process is subcritical, the multigeneration offsprings of an arbitrary member of the process die out in finitely many steps with probability 1. However, if $i\in I$ then we have $M^{(m_0+nm)}_{i,j}>0$ for every positive integer $n$. This implies that a member of type $i$ can have $n$-th generation offsprings with positive probability for any $n$. On the other hand, if $i\not\in I$ then it can be shown by standard calculation that we have $M^{(p)}_{i,j}=0$ for every state $j$. That is, in this case the multigeneration offsprings of a member of type $i$ die out at most in $p$ steps with probability 1.

Let $\eta_i(n)$ denote the $i$-th component of the vector $\bfeta(n)$, and consider the subpopulation of those members in the $n$-th generation of the process which are multigeneration offsprings of the innovations $\eta_i(k)$, $k=1,2,\dots$ It turns out that the existence of a stationary distribution requires that the size of these subpopulations converges in distribution as $n\to\infty$ for every $i$. If $i\in I$ then the corresponding state $j$ provides a feedback for the consecutive generations of the offsprings of $i$. Because of this feedback we need similar conditions on the innovation variable $\eta_i(1)$ as in the simple-type case. On the other hand, if $i\not\in I$ then such a feedback is not present, and the multigeneration offsprings of $i$ die out at most in $p$ steps. In this case the size of the subpopulations mentioned above is a stationary process without any additional condition on type $i$. This implies that the distributions of the innovation variables $\eta_i(1)$, $i\not\in I$, have no effect on the existence of a stationary distribution of the Galton--Watson process.

\begin{theorem}\label{thm:stat distr}
The subcritical Galton--Watson process $\bX_n$, $n\in\Z_+$, has a stationary distribution $\pi$ if and only if we have $\sum_{k=1}^\infty \log k P(\eta_i(1)=k)<\infty$ for every types $i\in I$.
\end{theorem}

Note that for an arbitrary nonnegative integer valued random variable $\zeta$ the sum $\sum_{k=1}^\infty \log k P(\zeta=k)$ is finite if and only if the expectation $E\log(\zeta+1)$ is finite.

Since $\cC$ is the only closed communication class by Theorem \ref{thm:irreducibility}, a subcritical Galton--Watson process has at most one positive recurrent class. This implies that the stationary distribution is unique and concentrated on $\cC$ if it exists. From Theorem \ref{thm:irreducibility} it also follows that every subcritical Galton--Watson process is $\psi$-\textit{irreducible} and \textit{aperiodic} in the sense of \citet{meyn09}. The \textit{maximal irreducibility measures} $\psi$ are those probability measures on the state space $\Z_+^p$ which are concentrated on $\cC$ and put positive mass at every state in this class. Furthermore, if the (unique) stationary distribution exists then the process is \textit{positive Harris recurrent}, and Theorem 13.0.1 of \citet{meyn09} implies that for any $\bx\in\Z_+^p$ we have
\[
\sup_{B\subseteq\Z_+^p} \big| P_\bx(\bX_n\in B) - \pi(B) \big| \to 0\,,
\qquad
n\to\infty\,.
\]
Under some stronger moment conditions we provide a rate for this convergence in Corollary \ref{cor:geom erg}.

Let $\widetilde\bX$ stand for a random vector with distribution $\pi$. In several applications showing the linear independence of the components of $\widetilde\bX$ is required. For example, assume that we want to estimate the mean matrix $\bM$ based on some observations $\bX_0,\dots,\bX_n$ by using the conditional least squares method or its weighted variant. Unfortunately, these estimators may not exist for every realizations of the sample. However, \citet{nedenyi15} showed that they are well-defined with asymptotic probability $1$ as $n\to\infty$ if the components of $\widetilde\bX$ are linearly independent. Also, a similar problem arose in \citet{pap13} about the maximum likelihood estimation of the parameters of the $\textrm{INAR}(p)$ process. Since the stationary distribution $\pi$ puts positive mass at every state in $\cC$, the components of $\widetilde\bX$ are linearly dependent if and only if the class $\cC$ is a subset of a lower dimensional affine subspace of $\R^p$. In our next theorem we provide necessary and sufficient conditions for this behavior.

We say that an arbitrary type $i$ \textit{dies out} if we have $P(\eta_j(1)=0)=1$ for every types $j$ for which there exists a nonnegative integer $n$ such that $M^{(n)}_{j,i}>0$. This property means that there is no innovation in type $i$, and a member of type $i$ can not be a multigeneration offspring of the innovations $\bfeta(1),\bfeta(2),\dots$ This implies that all members in $X_{n,i}$ are $n$-th generation offsprings of the initial population $\bX_0$. Since the process is subcritical, if a type dies out in the sense of our definition, then this type vanishes from the population in finitely many steps with probability one.

\begin{theorem}\label{thm:lin dep}
Assume that the Galton--Watson process $\bX_n$, $n\in\Z_+$, is subcritical. The communication class $\cC$ defined in Theorem \ref{thm:irreducibility} is a subset of a lower dimensional affine subspace of $\R^p$ if and only if either of the following conditions holds:
\begin{enumerate}

\item\label{thm:lin dep 1}
Any of the types dies out.

\item\label{thm:lin dep 2}
There exists a vector $\bc\in\R^p$, $\bc\neq\bnull$, such that $\bc^\top\bxi_i(1,1)=0$ almost surely for every types $i$, and the variable $\bc^\top\bfeta(1)$ is degenerate.

\end{enumerate}
If type $i$ dies out, then $\cC$ is a subset of the linear space defined by the equation $\be_i^\top\bx=0$, $\bx\in\R^p$. If \ref{thm:lin dep 2} holds, then $\cC$ is a subset of the affine subspace $\bc^\top\bx=\bc^\top\bfeta(1)$, $\bx\in\R^p$.
\end{theorem}

In our next theorem we investigate the moments of the stationary distribution $\pi$. For this goal let $\cZ$ be the set of those states $\bx=(x_1,\dots,x_p)^\top\in\Z_+^p$ for which we have $x_i=0$ for every type $i$ that dies out. Note that a type that dies out can not be an offspring of any type that does not die out, and there is no innovation in types that die out. This implies that $\cZ$ is a closed subset of the state space  in the sense that $P_\bx(\bX_1\in\cZ)=1$ for every $\bx\in\cZ$. Also, by Theorem \ref{thm:lin dep} we have $\cC\subseteq\cZ$, but the two sets may not coincide. For any real value $\alpha>0$ consider the set
\[
\cF_\alpha=\big\{f: \Z_+^p\to\R : |f(\bx)|\leq \|\bx\|^\alpha+1, \bx\in\Z_+^p \big\}\,.
\]

\begin{theorem}\label{thm:stationary}
Assume that the subcritical Galton--Watson process $\bX_n$, $n\in\Z_+$, has a stationary distribution $\pi$, and consider a real value $\alpha>0$. Then, the following statements are equivalent:
\begin{enumerate}

\item\label{thm:stationary 1}
The distribution $\pi$ has finite moment of order $\alpha$, that is, $\int_{\Z_+^p} \|\by\|^\alpha \pi(d\by)<\infty$.

\item\label{thm:stationary 2}
We have $E\|\bfeta(1)\|^\alpha<\infty$ and $E\|\bxi_i(1,1)\|^\alpha<\infty$ for all types $i$ that does not die out.

\end{enumerate}
Furthermore, if \ref{thm:stationary 1} or \ref{thm:stationary 2} is satisfied then there exist finite constants $a_1>1$ and $a_2>0$ such that
\begin{equation}\label{eq:meyn}
\sum_{n=0}^\infty a_1^n
\sup_{f\in\cF_\alpha} \bigg| E_\bx f(\bX_n) - \int_{\Z_+^p} f(\by)\pi(d\by) \bigg|
\leq a_2\big( \|\bx\|^\alpha+1\big)\,,
\end{equation}
for every state $\bx\in\cZ$. If, additionally, $E\|\bxi_i(1,1)\|^\alpha<\infty$ for every type $i$, then \eqref{eq:meyn} holds for every $\bx\in\Z_+^p$
\end{theorem}

It is a consequence of the theorem that the supremum in formula \eqref{eq:meyn} is of rate $o(1/a_1^n)$ as $n\to\infty$, implying that $E_\bx f(\bX_n)$ converges to $\int_{\Z_+^p} f(\by)\pi(d\by)$ for every $f\in\cF_\alpha$ at exponential rate. Since the function $f(\bx)=\|\bx\|^\beta$, $\bx\in\Z_+^p$, is an element of $\cF_\alpha$ for every $\beta\in[0,\alpha]$, we also obtain that
\[
\sum_{n=0}^\infty a_1^n
\sup_{\beta\in[0,\alpha]} \bigg| E_\bx\|\bX_n\|^\beta - \int_{\Z_+^p} \|\by\|^\beta \pi(d\by) \bigg| \leq a_2\big( \|\bx\|^r+1\big)\,,
\quad
n\to\infty\,.
\]
This means that those moments of the process which are of order at most $\alpha$ converge uniformly to the related moments of the stationary distribution.

Let us note that the finiteness of the mean matrix $\bM$ implies that $E\|\bxi_i(1,1)\|^\alpha<\infty$ holds for every type $i$ and for every $\alpha\in(0,1]$. Also, the indicator function $\1_B$ of an arbitrary set $B\subseteq\Z_+^p$ is an element of $\cF_\alpha$, and we have $E_\bx \1_B(\bX_n)=P_\bx(\bX_n\in B)$ and $\int_{\Z_+^p} \1_B(\by)\pi(d\by)=\pi(B)$. These facts along with Theorem \ref{thm:stationary} immediately implies the following statement.

\begin{corollary}\label{cor:geom erg}
Assume that the Galton--Watson process $\bX_n$, $n\in\Z_+$, is subcritical and $E\|\bfeta(1)\|^\alpha<\infty$ with some $\alpha>0$. Then,
\[
\sum_{n=0}^\infty a_1^n
\sup_{B\subseteq\Z_+^p} \big| P_\bx(\bX_n\in B) - \pi(B) \big|<\infty\,,
\qquad
\bx\in\Z_+^p\,,
\]
meaning that the process is geometrically ergodic.
\end{corollary}

As a final remark we note that inequality \eqref{eq:meyn} can be stated in an unconditional form too, where the initial value of the process is not fixed. If condition \ref{thm:stationary 2} of Theorem \ref{thm:stationary} is satisfied, the distribution of $\bX_0$ is concentrated to the set $\cZ$, and $E\|\bX_0\|^\alpha<\infty$, then by conditioning with respect to $\bX_0$ we obtain the inequality
\[
\sum_{n=0}^\infty a_1^n
\sup_{f\in\cF_\alpha} \bigg| Ef(\bX_n) - \int_{\Z_+^p} f(\by)\pi(d\by) \bigg|
\leq a_2\big( E\|\bX_0\|^\alpha+1\big)<\infty\,.
\]
Furthermore, if all of the offspring variables have finite moment of order $\alpha$, then the restriction of the initial distribution to $\cZ$ can be omitted. Using this result one can show the uniform convergence of the unconditional probabilities $P(\bX_n\in B)$, $B\subseteq\Z_+^p$, similarly as of the conditional ones in Corollary \ref{cor:geom erg}.


\section{Proofs}

In this section we present the proofs of the results stated in Section \ref{sec:main results}. The first statement is a fundamental observation in the subject of our paper.

\begin{proposition}\label{prop:eigenvector}
Let $\bA\in\R^{p\times p}$ be a matrix having only non-negative entries. If $\varrho(\bA)<1$, then there exist a constant $\lambda\in(0,1)$ and a vector $\bv\in\R^p$ such that all components of $\bv$ are strictly positive and $\bA\bv\leq \lambda\bv$.
\end{proposition}

\begin{proof}
Since the eigenvalues are continuous functions of the matrix entries, there exists an $\varepsilon>0$ such that $\lambda:=\varrho(\bA+\varepsilon)<1$. Then, the Perron--Frobenius theorem implies that the positive matrix $\bA+\varepsilon$ has an eigenvector $\bv$ with eigenvalue $\lambda$ such that all components of $\bv$ are strictly positive. With this vector we get the inequality $\bA\bv\leq (\bA+\varepsilon)\bv=\lambda\bv$.
\end{proof}

\begin{proof}[Proof of Theorem \ref{thm:irreducibility}]
We will use the representation of the multitype Galton--Watson process $\bX_n$, $n\in\Z_+$, provided by Section 2.7 of \citet{mode71}. Let the vectors $\bY_n$ and $\bV_{k+n}(k)$, $n,k=1,2,\dots$, stand for the number of the $n$-th generation offsprings of the initial population $\bX_0$ and of the innovation variable $\bfeta(k)$, respectively. Also, let $\bY_0:=\bX_0$ and $\bV_n(n):=\bfeta(n)$ for every $n$. Then,  we obtain the representation of \citet{mode71} in the form
\begin{equation}\label{eq:reprezentation}
\bX_n=\bY_n+\bZ_n:=\bY_n+\bV_n(1)+\cdots+\bV_n(n)\,,
\qquad
n=1,2,\dots,
\end{equation}
and the independence of the variables in \eqref{eq:ind variables} implies that $\bY_n,\bV_n(1),\dots,\bV_n(n)$ are independent of each other. (This equation can be proved by standard calculations too, by showing that the probability generation function of $\bX_n$ is equal to the product of the probability generating functions of the variables on the right side.) From the definitions of the variables it follows that the sequences $\bZ_n=\bV_n(1)+\cdots+\bV_n(n)$, $n=1,2,\dots$, and $\bY_n$, $n\in\Z_+$, are independent of each other. Let us note that $\bY_n$, $n\in\Z_+$, is a multitype Galton--Watson process without immigration, which implies that this process becomes extinct in finitely many steps with probability 1 in case of any initial distribution.

For $n=1,2,\dots$ let $\cD_n\subseteq\Z_+^p$ denote the range of the variable $\bZ_n$, that is, the set of the states $\bx\in\Z_+^p$ for which $P(\bZ_n=\bx)>0$. Since $\bZ_n$ and $\bZ_{n+1}$ are independent of $\bX_0$, and on the event $\{\bX_0=\bnull\}$ we have $\bX_n=\bZ_n$ and $\bX_{n+1}=\bZ_{n+1}$, we get that
\[
\begin{split}
0 & =P\big(\bZ_{n+1}\not\in\cD_{n+1}\mid\bX_0=\bnull \big)
=P\big(\bX_{n+1}\not\in\cD_{n+1}\mid \bX_0=\bnull\big) \\
& =\sum_{\bx\in\Z^p_+} P\big(\bX_{n+1}\not\in\cD_{n+1}\mid \bX_n=\bx, \bX_0=\bnull \big)
P\big(\bX_n=\bx\mid \bX_0=\bnull\big) \\
& = \sum_{\bx\in\cD_n} P_\bx(\bX_1\not\in\cD_{n+1}) P(\bZ_n=\bx)\,.
\end{split}
\]
Because the terms of the last sum are nonnegative, it follows from the definition of $\cD_n$ that $P_\bx(\bX_1\not\in\cD_{n+1})=0$ for every $\bx\in\cD_n$. It is a consequence that the set $\cC_n:=\cup_{k=n}^\infty \cD_k\subseteq\Z_+^p$ is closed for any $n$ in the sense that $P_\bx(\bX_1\in\cC_n)=1$ holds for every $\bx\in\cC_n$.

Let us recall that the sequence $\bY_n$, $n\in\Z_+$, is a subcritical Galton--Watson process without immigration, which implies that $P_\bx(\bY_n=\bnull)\to 1$ as $n\to\infty$ for any $\bx\in\Z_+^p$. This means that there exists an integer $n^*(\bx)$ such that $P_\bx(\bY_n=\bnull)>0$ holds for every $n\geq n^*(\bx)$. Because the number of the multigeneration offsprings of the members of the initial population are independent of each other, we obtain that
\[
P_\bx(\bY_n=\bnull)= P_{\be_1}(\bY_n=\bnull)^{x_1}\cdots P_{\be_p}(\bY_n=\bnull)^{x_n}>0
\]
for every $n\geq n^*:=\max(n^*(\be_1),\dots,n^*(\be_p))$. That is, the process $\bY_n$, $n\in\Z_+$, dies out in $n^*$ steps with positive probability in case of any initial state $\bx$.

Consider an arbitrary integer $n\geq n^*$ and states $\bx\in\Z_+^p$, $\bz\in\cD_n$. The independence of $\bZ_n$ of the variables $\bX_0=\bY_0$ and $\bY_n$ implies the inequality
\begin{equation}\label{eq:communicate}
P_\bx(\bX_n=\bz)\geq P_\bx \big(\bY_n=\bnull,\bZ_n=\bz\big)=P_\bx(\bY_n=\bnull) P(\bZ_n=\bz)>0\,,
\end{equation}
meaning that every elements of $\cD_n$ are accessible from the arbitrary state $\bx$ in $n$ steps. It is a consequence that the elements of the set $\cC:=\cC_{n^*}$ communicate with each other. Since $\cC$ is closed, it is a communication class of the process $\bX_n$, $n\in\Z_+$. Consider an arbitrary state $\bz\in\cD_n$ and a non-negative integer $m$, and let $\bx\in\Z_+^p$ be a state such that $P_\bz(\bX_m=\bx)>0$. Using equation \eqref{eq:communicate} again we find that
\[
P_\bz(\bX_{n+m}=\bz) \geq P_\bz(\bX_m=\bx)P_\bx(\bX_n=\bz)>0\,,
\]
that is, state $\bz$ is accessible from itself in $n+m$ steps. Since $m$ was arbitrary non-negative integer, the communication class $\cC$ is aperiodic.

To prove the theorem it is only remained to show that the process $\bX_n$, $n\in\Z_+$, reaches the class $\cC$ in finitely many steps with probability 1 in case of any initial distribution. Because the state space is countable, it is enough to prove this statement under the condition $\{\bX_0=\bx\}$ where $\bx\in\Z_+^p$ is an arbitrary fixed state. Note that $\{\bY_n=\bnull\}$, $n\in\Z_+$, is an increasing sequence of events. From this we get that
\[
\begin{split}
& P_\bx\big( \exists n\geq n^*: \bX_n\in\cC\big)
\geq P_\bx\big( \exists n\geq n^*: \bX_n\in\cD_n\big)
\geq P_\bx\big( \exists n\geq n^*: \bY_n=\bnull\big) \\
& \qquad
= P_\bx \big( \cup_{n=n^*}^\infty \{\bY_n=\bnull\} \big)
= \lim_{n\to\infty} P_\bx(\bY_n=\bnull)=1\,,
\end{split}
\]
which completes the proof.
\end{proof}

In the next step we prove Theorem \ref{thm:stat distr}. For this goal we need some technical results stated in Propositions \ref{prop:compound}--\ref{prop:GW dominant}.

\begin{proposition}\label{prop:compound}
Consider independent and identically distributed nonnegative valued random variables $\xi_1,\xi_2\dots$ such that $0<E\xi_1<\infty$. Also, let $\eta$ be a nonnegative integer valued random variable being independent of $\xi_1,\xi_2\dots$. Then,
\[
E\log\Bigg(\sum_{k=1}^\eta \xi_k + 1\Bigg)<\infty
\qquad\textrm{if and only if}\qquad
E\log(\eta+1)<\infty.
\]
\end{proposition}

\begin{proof}
First, assume that $E\log(\eta+1)$ is finite. By conditioning with respect to $\eta$ and using Jensen's inequality to the logarithm function we get that
\[
\begin{split}
& E\log\Bigg(\sum_{k=1}^\eta \xi_k + 1\Bigg)
=E E \Bigg[ \log\Bigg(\sum_{k=1}^\eta \xi_k + 1\Bigg) \,\Big|\, \eta\Bigg]
\leq E\log\Bigg( E \Bigg[ \sum_{k=1}^\eta \xi_k \,\Big|\, \eta \Bigg] + 1\Bigg) \\
& = E\log \big( \eta E\xi_1 + 1\big)
\leq
\begin{cases}
E\log( \eta+ 1)<\infty, & \textrm{if } E\xi_1\leq 1, \\
\log E\xi_1 + E\log(\eta+ 1)<\infty, & \textrm{if } E\xi_1\geq 1.
\end{cases}
\end{split}
\]

Now, consider the case when $E\log(\eta+1)$ is infinite. By the assumptions there exists a constant $c\in(0,1)$ such that $p:=P(\xi_1\geq c)>0$. Since we have $\xi_k\geq c\1_{\{\xi_k\geq c\}}$ with probability 1 for every $k$, we obtain the inequalities
\begin{equation}\label{eq:log compound}
E\log\Bigg(\sum_{k=1}^\eta \xi_k + 1\Bigg)
\geq E \log\Bigg(\sum_{k=1}^\eta c \1_{\{\xi_k\geq c\}} + 1\Bigg)
\geq \log c + E \log\Bigg(\sum_{k=1}^\eta \1_{\{\xi_k\geq c\}} + 1\Bigg)\,.
\end{equation}
Let $\zeta_n$ stand for a random variable having binomial distribution with parameters $n$ and $p$. Chebishev's inequality implies that
\[
P\big(\zeta_n\geq np-n^{1/2}\big)
\geq P\big( |\zeta_n-E\zeta_n|\leq n^{1/2}\big)
\geq 1-\frac{np(1-p)}{n}\geq \frac{3}{4}.
\]
Because the conditional distribution of the sum $\sum_{k=1}^\eta \1_{\{\xi_k\geq c\}}$ with respect to the event $\{\eta=n\}$ is the same as the law of $\zeta_n$ we get that
\begin{equation}\label{eq:binom ineq}
E \log\Bigg(\sum_{k=1}^\eta \1_{\{\xi_k\geq c\}} + 1\Bigg)
= \sum_{n=0}^\infty E\log(\zeta_n+1)P(\eta=n)
\geq \frac{3}{4} \sum_{n=0}^\infty \log\big(np-n^{1/2}+1\big) P(\eta=n).
\end{equation}
If $n$ is large enough, then we have
\[
\log\big(np-n^{1/2}+1\big)\geq \log\big((n+1)p/2\big)=\log(n+1)+\log(p/2).
\]
Since the expectation $E\log(\eta+1)$ is infinite, the sum on the right side of \eqref{eq:binom ineq} is divergent. Then, the proposition is proved by inequality \eqref{eq:log compound}.
\end{proof}

\begin{proposition}\label{prop:Markov dominant}
Let $\bX_n$ and $\bX'_n$, $n\in\Z_+$, be irreducible time-homogeneous Markov chains on some state spaces $\cC\subseteq\Z_+^p$ and $\cC'\subseteq\Z_+^p$. Assume that there exist states $\bx_0\in\cC$, $\bx'_0\in\cC'$, $\bx_0\leq\bx'_0$, such that the variables $\bX_1,\bX_2,\dots$ are conditionally independent of $\bX'_0$ with respect to the event $\{\bX_0=\bx_0\}$, and $\bX'_1,\bX'_2,\dots$ are conditionally independent of $\bX_0$ with respect to $\{\bX'_0=\bx'_0\}$. Furthermore, assume that
\begin{equation}\label{eq:Markov dominant}
P\big(\bX_n\leq\bX'_n\mid \bX_0=\bx_0,\bX'_0=\bx'_0\big)=1,
\qquad
n=1,2,\dots
\end{equation}
Then the following holds:
\begin{enumerate}

\item
If $\bX'_n$, $n\in\Z_+$, is recurrent, then $\bX_n$, $n\in\Z_+$, is recurrent.

\item
If $\bX'_n$, $n\in\Z_+$, is positive recurrent, then $\bX_n$, $n\in\Z_+$, is positive recurrent.

\end{enumerate}
\end{proposition}

\begin{proof}
Let $p_\bX^{(n)}(\cdot,\cdot)$ and $p_{\bX'}^{(n)}(\cdot,\cdot)$ denote the $n$-step transition probabilities of the processes, and let $P_A$ stand for the conditional probability with respect to $A:=\{\bX_0=\bx_0,\bX'_0=\bx'_0\}$. Also, introduce the set $\cC_0:=\{\bx\in\cC: \bx\leq\bx'_0\}$. In our proof we will use the well-know characterization of the types of states of Markov chains based on the asymptotic behavior of the transition probabilities. (See the main theorem in Section XV.5 of \citet{feller68}, for example.)

If the chain $\bX'_n$, $n\in\Z_+$, is recurrent then the assumptions and the characterization of recurrent states imply that
\[
\sum_{\bx\in\cC_0} \sum_{n=0}^\infty p_\bX^{(n)} (\bx_0,\bx)
= \sum_{n=0}^\infty P_A(\bX_n\in\cC_0)
\geq \sum_{n=0}^\infty P_A(\bX'_n=\bx'_0)
= \sum_{n=0}^\infty p_{\bX'}^{(n)} (\bx'_0,\bx'_0)
=\infty\,.
\]
Hence, there exists a state $\bx^*\in\cC_0$ such that $\sum_{n=0}^\infty p_\bX^{(n)} (\bx_0,\bx^*)=\infty$. Since the process $\bX_n$, $n\in\Z_+$, is irreducible, we have $p_\bX^{(k)}(\bx^*,\bx_0)>0$ for some $k\in\Z_+$. This leads to the inequality
\[
\sum_{n=0}^\infty p_\bX^{(n+k)} (\bx_0,\bx_0)
\geq \sum_{n=0}^\infty p_\bX^{(n)} (\bx_0,\bx^*) p_\bX^{(k)} (\bx^*,\bx_0)=\infty\,,
\]
meaning that $\bx_0$ is a recurrent state of the chain $\bX_n$, $n\in\Z_+$, and the first statement is proved.

Similarly, if the process $\bX'_n$, $n\in\Z_+$, is positive recurrent, then
\[
\begin{split}
\sum_{\bx\in\cC_0} \limsup_{n\to\infty} p_\bX^{(n)} (\bx_0,\bx)
& \geq \limsup_{n\to\infty} P_A(\bX_n\in\cC_0)
\geq \limsup_{n\to\infty} P_A(\bX'_n=\bx'_0) \\
& = \limsup_{n\to\infty} p_{\bX'}^{(n)} (\bx'_0,\bx'_0)>0\,.
\end{split}
\]
This implies that $\limsup_{n\to\infty} p_\bX^{(n)} (\bx_0,\bx^*)>0$ for some state $\bx^*\in\cC_0$. Again, if $k\in\Z_+$ is a constant such that  $p_\bX^{(k)}(\bx^*,\bx_0)>0$, then
\[
\limsup_{n\to\infty} p_\bX^{(n+k)} (\bx_0,\bx_0)
\geq \limsup_{n\to\infty} p_\bX^{(n)} (\bx_0,\bx^*) p_\bX^{(k)} (\bx^*,\bx_0) >0\,.
\]
From this inequality the characterization of the types implies that the chain $\bX_n$, $n\in\Z_+$, is positive recurrent, and the proof is complete.
\end{proof}

\begin{proposition}\label{prop:GW dominant}
Consider subcritical $p$-type Galton--Watson processes $\bX_0,\bX_1,\dots$ and $\bX'_0,\bX'_1,\dots$ based on the offspring and innovation vectors $\bxi_i(n,k)$, $\bfeta(n)$ and $\bxi'_i(n,k)$, $\bfeta'(n)$, $i=1,\dots,p$, $n,k=1,2,\dots$, respectively. Assume that all of these offspring and innovation vectors are independent of $\bX_0$  and $\bX'_0$, and assume that $\bxi_i(n,k)\leq\bxi'_i(n,k)$ and $\bfeta(n)\leq\bfeta'(n)$ hold for every $i$, $n$ and $k$ with probability 1. Under these assumptions if the process $\bX'_0,\bX'_1,\dots$ has a stationary distribution, then $\bX_0,\bX_1,\dots$ has a stationary distribution, as well.
\end{proposition}

\begin{proof}
Let $\cC$ and $\cC'$ stand for the unique closed communication classes of the processes provided by Theorem \ref{thm:irreducibility}, and consider arbitrary values $\bx_0,\bx'_0\in\Z_+^p$ such that $\bx_0\leq\bx'_0$. By the branching mechanism the processes satisfy the independence conditions of Proposition \ref{prop:Markov dominant}, and it can be shown by recursion with respect to $n$ that \eqref{eq:Markov dominant} also holds. Since the process $(\bX_n,\bX'_n)$, $n\in\Z_+$, reaches the set $\cC\times\cC'$ in finitely many steps with probability 1, formula \eqref{eq:Markov dominant} implies that there exists states $\bx\in\cC$ and $\bx'\in\cC'$ such that $\bx\leq\bx'$. This means that the initial values $\bx_0$ and $\bx'_0$ can be chosen as elements of these classes, respectively. Since $\cC$ and $\cC'$ are closed, we can restrict the processes to these sets, resulting that the restricted processes satisfy all conditions of Proposition \ref{prop:Markov dominant}.

If $\bX'_0,\bX'_1,\dots$ has a stationary distribution, then this distribution must be concentrated to $\cC'$, the only closed communication class. This implies that the restriction of the process to $\cC'$ is a positive recurrent Markov chain. Then, by Proposition \ref{prop:Markov dominant} the process $\bX_0,\bX_1,\dots$ is positive recurrent on $\cC$, as well, and by the theory of Markov chains the latter process has a stationary distribution.
\end{proof}

\begin{remark}\label{remark:no stat disrt}
Consider a single-type Galton--Watson process $X'_0,X'_1,\dots$ defined by
\begin{equation}\label{eq:single-type def}
X'_n=\sum_{k=1}^n \xi'(n,k)+\eta'(n),
\qquad
n=1,2,\dots
\end{equation}
such that $E\xi'(1,1)>0$ and $E\log(\eta'(1)+1)=\infty$. Then, it can be shown by using some classical results on Galton--Watson processes that $X'_0,X'_1,\dots$ does not have any stationary distribution. For example, this result was proved in the subcritical case by \citet{foster71} in their Corollary 2. To illustrate the application of our Proposition \ref{prop:GW dominant} we present a short and simple proof for the remaining cases.

Let $X_0,X_1,\dots$ denote the single-type Galton--Watson process corresponding to the initial value $X_0:=X'_0$ and to the offspring and innovation variables $\xi(n,k):=\1_{\{\xi'(n,k)\geq 1\}}$ and $\eta(n):=\eta'(n)$, $n,k=1,2,\dots$ This process is defined by replacing the vectors $\xi'(n,k)$ and $\eta'(n)$ in the recursion \eqref{eq:single-type def} by $\xi(n,k)$ and $\eta(n)$, respectively. Note that the processes $X_n$ and $X'_n$, $n\in\Z_+$, satisfy the assumptions of Proposition \ref{prop:GW dominant}, and the moment conditions imply that $E\xi(1,1)>0$ and $E\log(\eta(1)+1)=\infty$. If $\xi(1,1)=0$ with positive probability, then $E\xi(1,1)<1$, meaning that $X_0,X_1,\dots$ is subcritical. From this the referred result of \citet{foster71} implies that $X_0,X_1,\dots$ does not have any stationary distribution. If $\xi(1,1)=1$ with probability 1, then $X_n\to\infty$ almost surely as $n\to\infty$, resulting that $X_0,X_1,\dots$ does not have any stationary distribution in this case neither. Then, by using Proposition \ref{prop:GW dominant} we immediately obtain the statement.
\end{remark}

\begin{proof}[Proof of Theorem \ref{thm:stat distr}]
First, we show the existence of a stationary distribution under the logarithmic moment condition of the theorem in the case when the set $I$ contains all types. Let us recall that the eigenvalues are continuous functions of the matrix entries. Since the process $\bX_n$, $n\in\Z_+$, is subcritical, there exists an $\varepsilon>0$ such that $\varrho(\bM')<1$ with $\bM':=\bM+\varepsilon$. Let $\bfone\in\R^p$ denote the vector whose components are equal to $1$, and consider random vectors $\1_i(n,k)$, $i=1,\dots,p$, $n,k=1,2,\dots$, being independent of each other and of the variables in formula \eqref{eq:ind variables} and having common distribution $P(\1_i(n,k)=\bfone)=\varepsilon$ and $P(\1_i(n,k)=\bnull)=1-\varepsilon$. Additionally, consider the multitype Galton--Watson process $\bX'_n$, $n\in\Z_+$, defined by replacing the variables \eqref{eq:ind variables} in the recursion \eqref{eq:GW def} by the the initial value $\bX'_0:=\bX_0$ and by the offspring and innovation variables
\[
\bxi'_i(n,k):=\bxi_i(n,k)+\1_i(n,k)\,,
\qquad
\bfeta'(n):=\bfeta(n)\,,
\qquad
i=1,\dots,p\,,
\quad
n,k=1,2,\dots
\]
Note that the mean matrix of the new process is $E[\bxi'_1(1,1),\dots, \bxi'_p(1,1)]^\top =\bM'$, where $\bM'$ is a positive matrix and $\rho(\bM')<1$. For any $a_1,\dots,a_p\geq 0$ we have the algebraic inequality
\[
\log(a_1+1)+\cdots+\log(a_n+1)
=\log\big((a_1+1)\cdots(a_n+1)\big)
\geq \log(a_1+\cdots+a_n+1)\,.
\]
From this we obtain that
\begin{equation}\label{eq:log moment}
E\log\big(\|\bfeta(1)\|+1\big)\leq \sum_{i=1}^p E\log\big(\eta_i(1)+1\big)<\infty\,.
\end{equation}
meaning that $\sum_{k=1}^\infty \log k P\big(\|\bfeta(1)\|=k\big)$ is finite. Then, Corollary 1 of \citet{kaplan73} implies that the process $\bX'_n$, $n\in\Z_+$, has a stationary distribution. Since the offspring and the innovation variables satisfy the conditions of our Proposition \ref{prop:GW dominant}, it follows that the process $\bX_n$, $n\in\Z_+$, has a stationary distribution too.

Now, assume that the logarithmic moment condition of the theorem holds, and consider the case when $I$ does not contain all types. Let $I^c$ stand for the complement of set $I$, and define the random vectors $\bfeta^I(n)$ and $\bfeta^{I^c}(n)$, $n=1,2,\dots$, by their $i$-th components
\[
\eta^I_i(n)
:=\begin{cases}
\eta_i(n), & i\in I, \\
0, & i\in I^c,
\end{cases}
\qquad
\eta^{I^c}_i(n)
:=\begin{cases}
0, & i\in I, \\
\eta_i(n), & i\in I^c,
\end{cases}
\qquad
i=1,\dots,p.
\]
Let $\bV^I_{k+n}(k)$ and $\bV^{I^c}_{k+n}(k)$, $n=0,1,\dots$, denote the number of the $n$-th generation offsprings of the innovation variables $\bfeta^I(k)$ and $\bfeta^{I^c}(k)$, respectively, and introduce the $\Z^p_+$-valued processes
\begin{equation}\label{eq:decomposition}
\bZ^I_n:=\bV^I_n(1)+\cdots+\bV^I_n(n),
\qquad
\bZ^{I^c}_n:=\bV^{I^c}_n(1)+\cdots+\bV^{I^c}_n(n),
\qquad
n=1,2,\dots
\end{equation}
Based on the construction, the sequences in \eqref{eq:decomposition} are multitype Galton--Watson processes corresponding to the initial states $\bZ^I_0:=\bZ^{I^c}_0:=\bnull$ and to the innovation variables $\bfeta^I(1),\bfeta^I(2),\dots$ and $\bfeta^{I^c}(1),\bfeta^{I^c}(2),\dots$, respectively, having the same offspring distributions as the original process $\bX_n$, $n\in\Z_+$. Also, we have $\bV^I_n(k)+\bV^{I^c}_n(k)=\bV_n(k)$ for every $n$ and $k$, implying the identity $\bZ_n=\bZ^I_n+\bZ^{I^c}_n$, $n\in\Z_+$.

Consider random pairs $(\bU^I_n,\bU^{I^c}_n)$, $n=1,2,\dots$, which are independent of each other and of the initial variable $\bX_0$ such that $(\bU^I_n,\bU^{I^c}_n)$ has the same distribution as $(\bV^I_n(1),\bV^{I^c}_n(1))$ for every $n$. Note that the pairs $(\bV^I_n(k),\bV^{I^c}_n(k))$, $k=1,\dots,n$, are independent of each other, and $(\bV^I_n(k),\bV^{I^c}_n(k))$ has the same distribution as $(\bV^I_{n-k}(1),\bV^{I^c}_{n-k}(1))$, respectively. Then, for every $n$ the joint distribution of $(\bV^I_n(k),\bV^{I^c}_n(k))$, $k=1,\dots,n$, is the same as the joint distribution of $(\bU^I_{n-k+1},\bU^{I^c}_{n-k+1})$, $k=1,\dots,n$. This implies that
\begin{equation}\label{eq:vector sum}
\myvector{\bZ^I_n}{\bZ^{I^c}_n}
= \sum_{k=1}^n \myvector{\bV^I_n(k)}{\bV^{I^c}_n(k)}
\eq{\cD} \sum_{k=1}^n \myvector{\bU^I_k}{\bU^{I^c}_k}
\to \sum_{k=1}^\infty \myvector{\bU^I_k}{\bU^{I^c}_k},
\qquad
n\to\infty,
\end{equation}
where the convergence is understood in almost sure sense.

Since all components of the innovation variable $\bfeta^I(1)$ have finite logarithmic moment by assumption, the first step of current proof implies that the Galton--Watson process $\bZ^I_n$, $n\in\Z_+$, has a (proper) stationary distribution. Additionally, the process converges to $\sum_{n=1}^\infty \bU^I_n$ in distribtution by formula \eqref{eq:vector sum}. From these we obtain that the law of the limit is the stationary distribution of the process, resulting that $\sum_{n=1}^\infty \bU^I_n$ is convergent with probability 1. Note that the multigeneration offsprings of every members of an arbitrary type $i\in I^c$ vanish at most in $p$ steps. This means that $\bU_n=_\cD\bV^{I^c}_n(1)=\bnull$ for any $n\geq p+1$.  Let us recall that $\bY_n$, $n\in\Z_+$, denotes that number of the $n$-th generation offsprings of the initial population $\bX_0$. Since this process dies out in finitely steps with probability 1 in case of any initial distribution, we obtain the almost sure convergence $\bY_n\to\bnull$, $n\to\infty$. Then, by using \eqref{eq:vector sum} we get that
\[
\bX_n=\bY_n+\bZ_n=\bY_n+\bZ^I_n+\bZ^{I^c}_n
\co{\cD} \bnull + \sum_{k=1}^\infty \bU^I_k + \sum_{k=1}^p \bU^{I^c}_k,
\qquad
n\to\infty,
\]
where the limit variable is finite with probability 1, and its law does not depend on the initial distribution. This convergence implies that the law of the limit variable is a stationary distribution for the Galton--Watson process $\bX_n$, $n\in\Z_+$, in case of an arbitrary set $I$.

We prove the contrary direction ot the theorem by contradiction. For this goal, assume that the process has a stationary distribution $\pi$, and there exist states $j_0,j$ and integers $m_0\geq 0$, $m\geq 1$ such that $M^{(m_0)}_{j_0,j}>0$, $M^{(m)}_{j,j}>0$ and $\sum_{k=1}^\infty \log k P(\eta_{j_0}(1)=k)=\infty$. Since the second inequality implies that $M^{(nm)}_{j,j}>0$ holds for any positive integer $n$, we can assume without the loss of generality that $m>m_0$. Additionally, it also follows that $E\log(\eta_{j_0}(1)+1)=\infty$.

For an arbitrary positive integer $n$ the members of the $nm$-th generation of the process can be divided into two groups. Some of the members are $m$-th generation offsprings of the population $\bX_{(n-1)m}$, and the others are members of the innovation $\bfeta(nm)$ or multigeneration offsprings of $\bfeta((n-1)m),\dots,\bfeta(nm-1)$. By the branching mechanism the distribution of the number of the members in the second group is the same as the distribution of $\bZ_m$. Also, the number of the $m$-th generation offsprings of an arbitrary member of type $i$ in the $(n-1)m$-th generation has the same distribution as the conditional law of $\bY_m$ with respect to $\{\bX_0=\be_i\}$. Furthermore, these variables are independent of each other and of $\bX_{(n-1)m}$. Now, consider random vectors $\bxi_i^{(m)}(n,k),\bfeta^{(m)}(n)$, $i=1,\dots,p$, $n,k=1,2,\dots$ being independent of each other and of $\bX_0$ such that $\bxi_i^{(m)}(n,k)$ has the same law as the conditional distribution of $\bY_m$ under $\{\bX_0=\be_i\}$, and $\bfeta^{(m)}(n)$ has the same distribution as $\bZ_m$, respectively. Then, we have
\begin{equation}\label{eq:multistep GW}
\bX_{nm}\eq{\cD}\sum_{i=1}^p \sum_{k=1}^{X_{(n-1)m,i}} \bxi_i^{(m)}(n,k) + \bfeta^{(m)}(n),
\qquad
n=1,2,\dots
\end{equation}
(We note that this equation can be proved by standard calculations too, by using generating functions.) In the following we assume that equation \eqref{eq:multistep GW} holds in almost sure sense for every $n$. We can do so without the loss of generality, because under this assumption the distribution of the process $\bX_{nm}$, $n\in\Z_+$, does not change.

Let $\bX^*_0,\bX^*_1,\dots$ stands for the $p$-type Galton--Watson process corresponding to the offspring and innovation vectors $\bxi^*_i(n,k):=\1_{\{i=j\}}\bxi^{(m)}_i(n,k)$, $\bfeta^*(n):=\bfeta^{(m)}(n)$, $i=1,\dots,p$, $n,k=1,2,\dots$, with the initial value $\bX^*_0:=\bX_0$. Then, we have $\bxi^*_i(n,k)\leq \bxi^{(m)}_i(n,k)$ and $\bfeta^*(n)\leq\bfeta^{(m)}(n)$ for every $i$, $n$ and $k$ with probability 1, and $(\bX_0,\bX^*_0)$ is independent of all of these offpsring and innovation variables. Since $\pi$ is a stationary distribution of the Markov chain $\bX_n$, $n\in\Z_+$, it is a stationary distribution of the subsequence $\bX_{nm}$, $n\in\Z_+$. Then, Proposition \ref{prop:GW dominant} implies that the process $\bX^*_n$, $n=0,1,\dots$, has a stationary distribution $\pi^*$ too. Observe that we have
\[
X^*_{n,j}=\sum_{k=1}^{X^*_{n-1,j}} \xi^*_{j,j}(n,k) + \xi^*_{j,j}(n,k),
\qquad
n=1,2,\dots,
\]
meaning that $X^*_{n,j}$, $n\in\Z_+$, is a single-type Galton--Watson process. Then, the $j$-th marginal of the measure $\pi^*$ is a stationary distribution for this Markov chain. However, in the next step we show that $X^*_{n,j}$, $n\in\Z_+$, does not have any stationary distribution, which leads to a contradiction. This proves that the moment condition of the theorem is necessary for the existence of a stationary distribution of the original process $\bX_n$, $n\in\Z_+$.

By definition $\eta_j^*(1)$ has the same distribution as the $j$-th component $V_{m,j}$ of the vector $\bV_m$, and $V_{m,j}$  denotes the total number of those elements in the $m$-th generation which are of type $j$, and also, members of the innovation $\bfeta(m)$ or multigeneration offsprings of $\bfeta(1),\dots,\bfeta(m-1)$. Note that $\eta:=\eta_{j_0}(m-m_0)$ is the number of those members in the innovation $\bfeta(m-m_0)$ which are of type $j_0$. Let $V$ stand for the number of those members in the $m$-th generation of the process which are of type $j$ and $m_0$-th generation offsprings of the population $\eta$. Consider random variables $\xi_1,\xi_2,\dots$ being independent of each other and of $\eta$ and having the same law as the conditional distribution $Y_{m,j}$ with respect to $\{\bX_0=\be_{j_0}\}$. Then, we have $V_{m,j}\geq V$, and by the branching mechanism the variable $V$ has the same distribution as the sum $\sum_{k=1}^\eta \xi_k$. Note that $E\xi_1=M_{j_0,j}^{(m_0)}\in(0,\infty)$ and $E\log(\eta+1)=\infty$ by assumption. From these Proposition \ref{prop:compound} implies that
\[
E\log\big(\eta_j^*(1)+1\big) = E\log\big(V_{m,j}+1\big)
\geq E\log\big(V+1\big) = E\log\bigg(\sum_{k=1}^\eta \xi_k+1\bigg)=\infty\,.
\]
Since the offspring variable $\xi^*_{j,j}(1,1)$ has the same distribution as $\xi^{(m)}_i(1,1)$, it follows that $E\xi^*_{j,j}(1,1)=M^{(m)}_{j,j}>0$. Then, by Remark \ref{remark:no stat disrt} the single-type Galton--Watson process $X^*_{n,j}$, $n\in\Z_+$, can not have any stationary distribution. This argument completes the proof of the theorem.
\end{proof}

\begin{proof}[Proof of Theorem \ref{thm:lin dep}]
First, we show that if either \ref{thm:lin dep 1} or \ref{thm:lin dep 2} is satisfied then $\cC$ is a subset of a lower dimensional affine subspace $\cS$ of $\R^p$. Assume that type $j$ dies out for some $j=1,\dots,p$. Then, by using the notations introduced in the proof of Theorem \ref{thm:irreducibility}, the $j$-th component of $\bV_n$ vanishes with probability 1 for every positive integer $n$. Since $\cC$ is defined as the union of the ranges of the variables $\bV_n$, $n\geq n^*$, the class $\cC$ is a subset of the linear subspace $\cS$ defined by the the equation $\be_j^\top\bv=0$, $\bv\in\R^p$.

Now assume that \ref{thm:lin dep 2} holds and consider an arbitrary $\bx'\in\cC$. Since $\cC$ is a communication class, state $\bx'$ is accessible from some $\bx\in\cC$ in one step. Working on the event $\{\bX_0=\bx\}$ we get the equation
\[
\bc^\top \bX_1
=\sum_{i=1}^p \sum_{k=1}^{x_i} \bc^\top\bxi_i(1,k) + \bc^\top\bfeta(1)
= 0 + \bc^\top\bfeta(1)\,.
\]
Because $P_\bx(\bX_1=\bx')>0$ and $\bc^\top\bfeta(1)$ is degenerate by assumption, the vector $\bx'$ is an element of the affine subspace $\cS$ defined by the equation $\bc^\top\bv=\bc^\top\bfeta(1)$, $\bv\in\R^p$.

For the contrary direction let $\cS\subsetneq\R^p$ denote the affine subspace generated by $\cC$, and assume that none of the types dies out. Consider an arbitrary state $\bx^*\in\cC$, and fix a vector $\by^*\in\Z_+^p$ such that $P(\bfeta(1)=\by^*)>0$. Since the set $\cV=\cS-\bx^*$ is a linear subspace of $\R^p$ with dimension less than $p$, the orthogonal complement $\cV^\perp$ of $\cV$ is a non-trivial linear subspace of $\R^p$, and we have $\bc^\top\bx=\bc^\top\bx^*$ for every $\bc\in\cV^\perp$ and $\bx\in\cS$.

Consider an arbitrary state $\bx\in\cC$ and an arbitrary vector $\bc\in\cV^\perp$, and work on the event $\{\bX_0=\bx\}$. The communication class $\cC$ is closed, which implies that the variable $\bX_1$ lies in $\cS$ with probability 1, and hence, we obtain the equation
\begin{equation}\label{eq:indep terms}
\bc^\top\bx^*=\bc^\top\bX_1
= \sum_{k=1}^{x_1} \bc^\top\bxi_1(1,k) + \cdots + \sum_{k=1}^{x_p} \bc^\top\bxi_p(1,k) + \bc^\top\bfeta(1)\,.
\end{equation}
Consider any type $i=1,\dots,p$. Since type $i$ does not die out by assumption, there exists a state $\bx\in\cC$ such that $x_i\neq 0$. With this state the left side of equation \eqref{eq:indep terms} is deterministic, and the terms on the right side are independent of each other, which imply that $\bc^\top\bxi_i(1,1)$ and $\bc^\top\bfeta(1)$ are degenerate variables. Then, by formula \eqref{eq:S def} we have
\begin{equation}\label{eq:S basic eq}
\bc^\top\bx^* =\bc^\top\bS(\bx)+\bc^\top\bfeta(1)
= E_\bx  \big(\bc^\top \bS(\bx) \big) +\bc^\top\by^*
= (\bM\bc)^\top\bx + \bc^\top\by^*
\end{equation}
with probability 1 for any state $\bx\in\cC$ and vector $\bc\in\cV^\perp$. Since $\cS$ is the affine subspace generated by the set $\cC$, equation \eqref{eq:S basic eq} is valid for any $\bx\in\cS$, as well. 

Consider an arbitrary vector $\bv\in\cV$, and note that both $\bv+\bx^*$ and $\bx^*$ are elements of $\cS$. Then, from equation \eqref{eq:S basic eq} it follows that
\[
(\bM\bc)^\top\bv
= (\bM\bc)^\top (\bv+\bx^*) - (\bM\bc)^\top \bx^*
= \bc^\top(\bx^*-\by^*) - \bc^\top(\bx^*-\by^*) =0\,.
\]
This implies that $\bM\bc\in\cV^\perp$ for any $\bc\in\cV^\perp$, and hence, we have $\psi(\cV^\perp)\subseteq\cV^\perp$ with the linear function $\psi:\R^p\to\R^p$, $\psi(\bc) = \bM\bc$. Because the process $\bX_n$, $n=0,1,\dots$, reaches the class $\cC$ in finitely many steps almost surely in case of any initial state, there exists a state $\bz\in\Z_+^p$, $\bz\not\in\cS$, such that the subspace $\cS$ is accessible from $\bz$ in one step. Since under the event $\{\bX_0=\bz\}$ the variable $\bX_1$ is an element of $\cS$ with positive probability, the equation
\begin{equation}\label{eq:eq for z}
\bc^\top\bx^* = \bc^\top\bX_1 = \bc^\top \bS(\bz)+\bc^\top\bfeta(1)
= E_\bz \big(\bc^\top \bS(\bz)\big) +\bc^\top\by^*
= (\bM\bc)^\top\bz + \bc^\top\by^*
\end{equation}
holds with positive probability in case of any $\bc\in\cV^\perp$. As a consequence, we get that $\bc^\top(\bx^*-\by^*) = (\bM\bc)^\top\bz$. Let us consider the orthogonal decomposition $\bz=\bx+\bx^\perp$ where $\bx\in\cS$ and $\bx^\perp\in \cV^\perp$, $\bx^\perp\neq\bnull$. From \eqref{eq:S basic eq} we obtain that
\[
\bc^\top(\bx^*-\by^*) = (\bM\bc)^\top \bz
= (\bM\bc)^\top \bx + (\bM\bc)^\top \bx^\perp
= \bc^\top(\bx^*-\by^*) + \psi(\bc)^\top\bx^\perp\,,
\]
and hence, $\psi(\bc)^\top\bx^\perp=0$ for any vector $\bc\in\cV^\perp$. That is, $\psi(\cV^\perp)\perp\bx^\perp\in\cV^\perp$ implying that $\psi(\cV^\perp)\subsetneq\cV^\perp$. This means that $\psi$ is not a full rank linear transformation, and there exists a vector $\bc^*\in\cV^\perp$ such that $\bM\bc^*=\psi(\bc^*)=\bnull$. Since the variables $\bc^\top\bxi_i(1,1)$, $i=1,\dots,p$, are deterministic in case of any $\bc\in\cV^\perp$, we get that
\[
(\bc^*)^\top\bxi_i(1,1)=E \big( (\bc^*)^\top\bxi_i(1,1) \big)
= (\bc^*)^\top E \bxi_i(1,1) =(\bM\bc^*)^\top \be_i=0\,,
\qquad
i=1,\dots,p\,,
\]
and the proof is complete.
\end{proof}

\begin{proof}[Proof of Theorem \ref{thm:stationary}]
First, assume that the stationary distribution $\pi$ has finite moment of order $\alpha$, and consider an arbitrary type $i$ that does not die out. Then, there exists a state $\bx\in\cC$ whose $i$-th component is not zero. If the initial distribution of the process is set to the stationary distribution $\pi$, then
\[
\infty> E\|\bX_1\|^\alpha \geq E_\bx \|\bX_1\|^\alpha P(\bX_0=\bx)
\geq E \big\|\bxi_i(1,1)\big\|^\alpha \pi\big(\{\bx\}\big)\,,
\]
proving that $E\|\bxi_i(1,1)\|^\alpha$ is finite. Similarly,
$E\|\bfeta(1)\|^\alpha\leq E\|\bX_1\|^\alpha<\infty$.

For the contrary direction assume that \ref{thm:stationary 2} holds, and consider the constant $\lambda\in(0,1)$ and the vector $\bv\in\Z_+^p$ of Proposition \ref{prop:eigenvector} with $\bA=\bM$. Since the proposition remains true if we multiply $\bv$ with a positive number, we can assume that the components of $\bv$ are larger than 1. Also, introduce the function $V(\bx)=(\bv^\top\bx)^\alpha+1$, $\bx\in\Z_+^p$, and note that $\|\bx\|^\alpha+1\leq V(\bx)$ for every states $\bx$. Our goal is to prove that
\begin{equation}\label{eq:Foster}
E_\bx V(\bX_1) - V(\bx) \leq -c_1 V(\bx) + c_2\1_{\cZ'}(\bx)\,,
\qquad
\bx\in\cZ\,,
\end{equation}
where $c_1>0$ and $c_2<\infty$ are suitable real values, and $\1_{\cZ'}$ is the indicator function of a suitable finite set $\cZ'\subseteq\cZ$. From Theorem \ref{thm:irreducibility} it follows that the process $\bX_n$, $n\in\Z_+$, is $\psi$-\textit{irreducible} and \textit{aperiodic} in the sense of \citet{meyn09} on the reduced state space $\cZ$, and their Proposition 5.5.5 implies that the finite set $\cZ'$ is \textit{petite}. This means that if we can prove the Foster--Lyapunov type criteria \eqref{eq:Foster} for every states $\bx\in\cZ$ then Theorem 15.0.1 of \citet{meyn09} immediately implies statement \ref{thm:stationary 1} and inequality \eqref{eq:meyn} in the current theorem.

If $\alpha\leq 1$ then the function $t\mapsto t^\alpha$ is nonnegative and concave on the positive halfline. This implies that $(s+t)^\alpha\leq s^\alpha + t^\alpha$ for any $s,t\geq 0$. By using Jensen's inequality we get that
\[
\begin{split}
& E_\bx \big(\bv^\top\bX_1\big)^\alpha
\leq E \big( \bv^\top\bS(\bx) \big)^\alpha + E\big( \bv^\top\bfeta(1) \big)^\alpha
\leq \big( \bv^\top E\bS(\bx) \big)^\alpha + E\big(\|\bv\| \|\bfeta(1)\| \big)^\alpha \\
& = \big( \bv^\top \bM^\top\bx \big)^\alpha + \|\bv\|^\alpha E\|\bfeta(1)\|^\alpha
\leq \big( \lambda \bv^\top\bx\big)^\alpha + \|\bv\|^\alpha E\|\bfeta(1)\|^\alpha\,.
\end{split}
\]
Consider an arbitrary constant $c_1\in(0,1-\lambda^\alpha)$ and the finite set
\[
\cZ'=\bigg\{
\bx\in\cZ: V(\bx)< \frac{\|\bv\|^\alpha E\|\bfeta(1)\|^\alpha + 1}{1-\lambda^\alpha-c_1}
\bigg\}\,.
\]
Then, for any $\bx\in\cZ\setminus\cZ'$ we have
\[
E_\bx V(\bX_1)
\leq \lambda^\alpha V(\bx) + \|\bv\|^\alpha E\|\bfeta(1)\|^\alpha +1
\leq (1 -c_1) V(\bx)\,.
\]
This implies inequality \eqref{eq:Foster} with
\[
c_2=\max_{\bx\in\cZ'} \lambda^\alpha V(\bx) + \|\bv\|^\alpha E\|\bfeta(1)\|^\alpha +1
<\infty\,.
\]

In the case $\alpha>1$ let $\|\cdot\|_{L^\alpha}$ stand for the $L^\alpha$-norm of random variables. Fix an arbitrary state $\bx\in\cZ$, and introduce the random vectors $\overline\bX_1:=\bX_1-E_\bx \bX_1$,
\[
\overline\bxi_i(n,k):=\bxi_i(n,k)-E\bxi_i(n,k)\,,
\quad
\overline\bfeta(n):=\bfeta(n)-E\bfeta(n)\,,
\quad
i=1,\dots,p\,,
\
n,k=1,2,\dots
\]
Then, we have
\[
\begin{split}
& \big( E_\bx(\bv^\top \bX_1)^\alpha \big)^{1/\alpha}
= \big\| \bv^\top \bX_1 \big\|_{L^\alpha}
= \big\|\bv^\top\overline\bX_1 + E \big(\bv^\top\bS(\bx)\big)
+ E\big(\bv^\top\bfeta(1)\big)\big\|_{L^\alpha} \\
& \leq \big\|\bv^\top\overline\bX_1\big\|_{L^\alpha}
+ \big\| \bv^\top\bM^\top\bx \big\|_{L^\alpha}
+ \big\| \bv^\top E\bfeta(1) \big\|_{L^\alpha}
\leq \big\|\bv^\top\overline\bX_1\big\|_{L^\alpha} + \lambda \bv^\top\bx + \bv^\top E\bfeta(1).
\end{split}
\]
By using the Marcinkiewicz--Zygmund inequality (see Theorem 13 of \citet{marc37} or Theorem 10.3.2 of \citet{chow97}) we get that
\[
\big\|\bv^\top\overline\bX_1\big\|_{L^\alpha}^\alpha
= E \bigg| \sum_{i=1}^p \sum_{k=1}^{x_i} \bv^\top\overline\bxi_i(1,k) + \bv^\top\overline\bfeta(1) \bigg|^\alpha
 \leq C E \bigg[ \sum_{i=1}^p \sum_{k=1}^{x_i} \big| \bv^\top\overline\bxi_i(1,k) \big|^2
+ \big| \bv^\top\overline\bfeta(1) \big|^2 \bigg]^{\alpha/2},
\]
where $C$ is a suitable positive constant depending only on $\alpha$. Let us note that for arbitrary nonnegative real numbers $a_1,\dots,a_n$ we have
\[
\big(a_1+\cdots+a_n\big)^{\alpha/2}\leq n^{\beta-1} \big(a_1^{\alpha/2}+\cdots+a_n^{\alpha/2}\big),
\]
where the value $\beta$ also depends only on $\alpha$. This inequality follows with $\beta=1$ in the case $\alpha\leq 2$ from the fact that the function $t\mapsto t^{\alpha/2}$ is concave on the positive halfline, and with $\beta=\alpha/2$ in the case $\alpha>2$ from the power mean inequality. Then, we get that
\[
\big\|\bv^\top\overline\bX_1\big\|_{L^\alpha}^\alpha
\leq C \big( \|\bx\| + 1\big)^{\beta-1}
\bigg[ \sum_{i=1}^p \sum_{k=1}^{x_i} E \big| \bv^\top\overline\bxi_i(1,k) \big|^\alpha
+ E \big| \bv^\top\overline\bfeta(1) \big|^\alpha \bigg]
\leq C \big( \|\bx\| + 1\big)^\beta b_\alpha\,,
\]
where
\[
b_\alpha:=\max\Big\{E\big|\bv^\top \overline\bxi_1(1,1)\big|^\alpha, \dots,
E\big|\bv^\top \overline\bxi_p(1,1)\big|^\alpha, E\big|\bv^\top \overline\bfeta(1)\big|^\alpha \Big\}
<\infty\,.
\]
Consider any constant $c_1\in(0,1-\lambda^\alpha)$. Since $(1-c_1)^{1/\alpha}>\lambda$ and $\beta/\alpha<1$, it follows that
\[
\begin{split}
E_\bx V(\bX_1)
& = E_\bx \big(\bv^\top \bX_1\big)^\alpha + 1
\leq \Big[ (Cb_\alpha)^{1/\alpha} \big( \|\bx\| + 1\big)^{\beta/\alpha}
+ \lambda \bv^\top\bx + \bv^\top E\bfeta(1) \Big]^\alpha +1 \\
& \leq \big[ (1-c_1)^{1/\alpha} \bv^\top\bx \big]^\alpha
\leq (1-c_1) V(\bx) \,,
\end{split}
\]
where the second inequality holds for all except finitely many values $\bx\in\cZ$. Let $\cZ'$ stand for the set of those states $\bx\in\cZ$ for which the second inequality does not hold. Then we obtain \eqref{eq:Foster} in the case $\alpha>1$ with
\[
c_2=\max_{\bx\in\cZ'} \Big[ (Cb_\alpha)^{1/\alpha} \big( \|\bx\| + 1\big)^{\beta/\alpha}
+ \lambda \bv^\top\bx + \bv^\top E\bfeta(1) \Big]^\alpha +1 <\infty\,.
\]

If all offspring distributions have finite moment of order $\alpha$, then one can show by similar calculations that inequality \eqref{eq:Foster} holds for any $\bx\in\Z_+^p$ with suitable constants $c_1, c_2$ and with a finite set $\cZ'\subseteq\Z_+^p$. Then, again, Theorem 15.0.1 of \citet{meyn09} implies that \eqref{eq:meyn} holds not only on the set $\cZ$ but for every states $\bx\in\Z_+^p$. This argument completes the proof of our last theorem.
\end{proof}

\section*{Acknowledgement}

This research was supported by project TKP2021-NVA-09. Project no. TKP2021-NVA-09 has been implemented with the support provided by the Ministry of Innovation and Technology of Hungary from the National Research, Development and Innovation Fund, financed under the TKP2021-NVA funding scheme.



\begin{thebibliography}{}
\bibliographystyle{plainnat}

\bibitem[Athreya and Ney(1972)]{athreya72}
Athreya, K.B. and Ney, P. (1972)
Branching processes.
Springer, Berlin-Heidelberg-New York.

\bibitem[Chow and Teicher(1997)]{chow97}
Chow, Y.S. and Teicher, H. (1997)
Probability theory. Independence, interchangeability, martingales.
Springer-Verlag, New York.


\bibitem[Feller(1968)]{feller68}
Feller, W. (1968)
An introduction to probability theory and its applications. Vol. I.
John Wiley \& Sons, Inc., New York-London-Sydney.

\bibitem[Foster and Williamson(1971)]{foster71}
Foster, J.H.  and Williamson, J.A. (1971)
Limit theorems for the Galton-Watson process with time-dependent immigration.
\textit{Z. Wahrsch. Verw. Gebiete} \textbf{20} 227-235.

\bibitem[Harris(1948)]{harris48}
Harris, T.E. (1948)
Branching processes.
\textit{Ann. Math. Statist.} \textbf{19} 474-494.

\bibitem[Kaplan(1973)]{kaplan73}
Kaplan, N. (1973)
The multitype Galton-Watson process with immigration.
\textit{Ann. Probab.} \textbf{1} 947-953.

\bibitem[Klimko and Nelson(1978)]{klimko78}
Klimko, L.A. and Nelson, P.I. (1978)
On conditional least squares estimation for stochastic processes.
\textit{Ann. Stat.} \textbf{6} 629-642.

\bibitem[Marcinkiewicz and Zygmund(1937)]{marc37}
Marcinkiewicz, J. and Zygmund, A. (1937)
Sur les foncions independantes.
\textit{Fund. Math.} \textbf{28} 60--90.

\bibitem[Meyn and Tweedie(2009)]{meyn09}
Meyn, S.P. and Tweedie, R.L. (2009)
Markov chains and stochastic stability.
Cambridge University Press, Cambridge.

\bibitem[Mode(1971)]{mode71}
Mode, C.J. (1971)
Multitype branching processes, theory and applications.
American Elsevier Publishing Co., Inc., New York.

\bibitem[Ned\'enyi(2015)]{nedenyi15}
Ned\'enyi, F. (2015)
Conditional least squares estimators for multitype Galton--Watson processes.
\textit{Acta Sci. Math. (Szeged)} \textbf{81} 325-348.

\bibitem[Pap and T. Szabó(2013)]{pap13}
Pap, Gy., T. Szabó, T. (2013)
Change detection in INAR(p) processes against various alternative hypotheses.
\textit{Comm. Statist. Theory Methods} \textbf{42} 1386-1405.




\end{thebibliography}
\end{document}